\documentclass[a4paper,11pt]{amsart}
\usepackage[english]{babel}
\usepackage{amsmath,tikz-cd}
\usepackage{amssymb}
\usepackage{mathrsfs}
\usepackage{enumerate}
\usepackage{mathtools,yfonts}
\usepackage{tikz,tkz-euclide}
\usepackage{pgfplots}
\usepackage{hyperref}
\usetikzlibrary{arrows}
\usepackage{makecell}

\usepackage{tcolorbox}
\usepackage[]{algorithm2e}

\renewcommand{\P}{\mathbb P}
\newcommand{\R}{\mathbb R}

\DeclareMathOperator{\Bl}{Bl}
\DeclareMathOperator{\Cl}{Cl}

\DeclareMathOperator{\NE}{NE}

\DeclareMathOperator{\mult}{mult}

\DeclareMathOperator{\Exc}{Exc}

\DeclareMathOperator{\Eff}{Eff}
\DeclareMathOperator{\Nef}{Nef}
\DeclareMathOperator{\Mov}{Mov}
\DeclareMathOperator{\Pic}{Pic}

\DeclareMathOperator{\Cox}{Cox}

\renewcommand{\P}{\mathbb{P}}

\newtheorem{thm}{Theorem}[section]

\newtheorem{Lemma}[thm]{Lemma}
\newtheorem{Proposition}[thm]{Proposition}

\newtheorem{Corollary}[thm]{Corollary}

\theoremstyle{definition}

\newtheorem{Definition}[thm]{Definition}
\newtheorem{Remark}[thm]{Remark}

\hypersetup{pdfpagemode=UseNone}
\hypersetup{pdfstartview=FitH}

\graphicspath{{immagini/}}
\usepackage{pgf,tikz,pgfplots}
\pgfplotsset{compat=1.15}
\usepackage{mathrsfs}
\usetikzlibrary{arrows}

\begin{document}

\title{Cox rings of blow-ups of multiprojective spaces}

\author[Michele Bolognesi]{Michele Bolognesi}
\address{\sc Michele Bolognesi\\
IMAG - Universit\'e de Montpellier, CNRS\\
Place Eug\`ene Bataillon\\
34095 Montpellier Cedex 5\\ France}
\email{michele.bolognesi@umontpellier.fr}

\author[Alex Massarenti]{Alex Massarenti}
\address{\sc Alex Massarenti\\ Dipartimento di Matematica e Informatica, Universit\`a di Ferrara, Via Machiavelli 30, 44121 Ferrara, Italy}
\email{msslxa@unife.it}

\author[Elena Poma]{Elena Poma}
\address{\sc Elena Poma\\ IMAG - Universit\'e de Montpellier, Place Eug\`ene Bataillon, 34095 Montpellier Cedex 5, France and Dipartimento di Matematica e Informatica, Universit\`a di Ferrara, Via Machiavelli 30, 44121 Ferrara, Italy}
\email{elena.poma@umontpellier.fr, elena.poma@unife.it}

\date{\today}

\begin{abstract}
Let $X^{1,n}_r$ be the blow-up of $\mathbb{P}^1\times\mathbb{P}^n$ in $r$ general points. We describe the Mori cone of $X^{1,n}_r$ for $r\leq n+2$ and for $r = n+3$ when $n\leq 4$. Furthermore, we prove that $X^{1,n}_{n+1}$ is log Fano and give an explicit presentation for its Cox ring.
\end{abstract}

\maketitle

\setcounter{tocdepth}{1}

\tableofcontents

\section{Introduction}
\textit{Mori dream spaces}, introduced by Y. Hu and S. Keel in \cite{HK00}, are varieties whose total coordinate ring, called the \textit{Cox ring}, is finitely generated. The birational geometry of a Mori dream space is encoded in its cone of effective divisors together with a chamber decomposition on it, called \textit{Mori chamber decomposition} whose chambers correspond to the nef cones of the birational models of the Mori dream space.

Finite generation of the Cox rings of blow-ups of projective spaces and products of projective spaces has been extensively investigated \cite{Mu01}, \cite{CT06}, \cite{AC17}, \cite{LP17}, \cite{GPP23}. In \cite[Theorem 1.3]{CT06} A. M. Castravet and J. Tevelev proved that the blow-up of $(\mathbb{P}^n)^s$ in $r$ general points is a Mori dream space if and only if 
$$
\frac{1}{s+1}+\frac{1}{r-n-1}+\frac{1}{n+1} > 1.
$$
Indeed, when the above inequality is not satisfied the effective cones of these blow-ups are not finitely generated and the proof of this last fact relies on the symmetries of their Picard groups which carry a natural Weyl group action. For products of projective spaces with unbalanced dimensions this is not the case. In this paper we push a little further the investigation of the birational geometry of these varieties initiated by T. Grange, E. Postinghel and A. Prendergast-Smith in \cite{GPP23}.

Let $X^{1,n}_r$ be the blow-up of $\mathbb{P}^1\times\mathbb{P}^n$ in $r$ general points $p_1,\dots,p_r \in \mathbb{P}^1\times\mathbb{P}^n$. We will denote by $\pi:X^{1,n}_r\rightarrow \mathbb{P}^1\times\mathbb{P}^n$ the blow-down morphism and by $\pi_1:\mathbb{P}^1\times\mathbb{P}^n\rightarrow\mathbb{P}^1$, $\pi_2:\mathbb{P}^1\times\mathbb{P}^n\rightarrow\mathbb{P}^n$ the projections onto the factors. Moreover, let us denote by $\widetilde{\pi}_1,\widetilde{\pi}_2$ the morphisms from $X^{1,n}_r$ to $\P^1$ and $\P^n$ induced by the projections.  We summarize the situation in the following diagram
$$
\begin{tikzcd}
             & {X^{1,n}_r} \arrow[d, "\pi"] \arrow[ldd, "\widetilde{\pi}_1"', bend right] \arrow[rdd, "\widetilde{\pi}_2", bend left] &              \\
             & \mathbb{P}^1\times\mathbb{P}^n \arrow[ld, "\pi_1"'] \arrow[rd, "\pi_2"]                                                &              \\
\mathbb{P}^1 &                                                                                                                        & \mathbb{P}^n
\end{tikzcd}
$$
Let $\Pic(X^{1,n}_r)$ denote the Picard group of $X^{1,n}_r$ and $N_1(X^{1,n}_r)$ the $\R$-vector space of $1$-dimensional cycles modulo numerical equivalence. Throughout the paper we will denote by:
\begin{itemize}
\item[-] $H_1$ the pull-back of a point of $\mathbb{P}^1$ via $\widetilde{\pi}_1$;
\item[-] $H_2$ the pull-back of a hyperplane in $\mathbb{P}^n$ via $\widetilde{\pi}_2$;
\item[-] $E_i$ the exceptional divisor over $p_i$ for $i = 1,\dots,r$;
\end{itemize}
and by 
\begin{itemize}
\item[-] $h_1$ the class of a general fiber of $\widetilde{\pi}_2$;
\item[-] $h_2$ the class of a line in a general fiber of $\widetilde{\pi}_1$;
\item[-] $e_i$ the class of a line in the exceptional divisor $E_i$ for $i = 1,\dots, r$.
\end{itemize}
We have that $\Pic(X^{1,n}_r) = \mathbb{Z}[H_1,H_2,E_1,\dots,E_r]$ and $N_1(X^{1,n}_r) = \mathbb{Z}[h_1,h_2,e_1,\dots,e_r]$.

Our main results on the Mori cone of $X^{1,n}_r$ in Propositions \ref{Mori_n+1}, \ref{Mori_n+2}, \ref{Mori_n+3} can be summerized as follows:
\begin{thm}
The Mori cone of $X^{1,n}_r$ is given by
$$
\NE(X^{1,n}_r)=\langle h_{1}-e_{i},h_{2}-e_{i},e_{i}\rangle
$$
for all $r\leq n+1$. Furthermore, 
$$
\NE(X^{1,n}_{n+2})=\langle h_{1}-e_{i},h_{2}-e_{i},e_{i},h_{1}+nh_{2}-e_{1}-\cdots-e_{n+2}\rangle
$$  
and if $n\leq 4$ then
$$\NE(X^{1,n}_{n+3})=\langle h_{1}-e_{i},h_{2}-e_{i},e_{i},h_{1}+nh_{2}-e_{i_{1}}-\dots-e_{i_{n+2}}\rangle$$ 
for $i,i_{1},\dots,i_{n+2}\in\{1,\dots,n+3\}$.
\end{thm}

Recall that a normal and $\mathbb{Q}$-factorial projective variety $X$ is \textit{log Fano} if there exists an effective divisor $D\subset X$ such that $-(K_X+D)$ is ample and the pair $(X,D)$ is Kawamata log terminal. By \cite[Corollary 1.3.2]{BCHM10} log Fano varieties are Mori dream spaces. The converse does not hold
in general, and there are several criteria for a Mori dream space to be log Fano \cite{GOST15}.

Our main results in the case $r\leq n+1$ in Proposition \ref{logf} and Theorem \ref{Cox_n+1} can be summarized as follows:
\begin{thm}
For $r\leq n+1$ the variety $X_{r}^{1,n}$ is log Fano and hence a Mori dream space. Furthermore, the Cox ring of $X_{n+1}^{1,n}$ is given by
$$
\Cox(X_{n+1}^{1,n}) \cong \frac{\mathbb{C}[S_0,\dots,S_n,T_{1,1},T_{1,2},\dots,T_{n+1,1},T_{n+1,2}]}{\left\langle g_1,\dots,g_{n-1}\right\rangle}
$$
where $S_i$ is the section associated to $H_2-E_1-\dots-E_{i-1}-E_{i+1}-\dots -E_{n+1}$, $T_{i,1}$ is the section associated to $H_1-E_i$ and $T_{i,2}$ is the section associated to $E_i$, and 
$$
g_i = (\beta_k\alpha_j-\beta_j\alpha_k)T_{i,1}T_{i,2} + (\beta_i\alpha_k-\beta_k\alpha_i)T_{j,1}T_{j,2} + (\beta_j\alpha_i-\beta_i\alpha_j)T_{k,1}T_{k,2}
$$
where for $1\leq i\leq n-1$ we set $k = j+1 = i+2$. 
\end{thm}
Furthermore, in Section \ref{sectionmori} we explicitly describe the nef cones of $X^{1,n}_r$ for $r\leq n+2$, and of $X^{1,n}_{n+3}$ for $n\leq 4$, and in Section \ref{SCox_n+1} we give generators for the cones of movable divisors and of moving curves of $X^{1,n}_{n+1}$.

\subsection*{Acknowledgments}
The first named author is supported by and member of the ANR project \it FanoHK \rm ANR-20-CE40-0023. The first named author is member of the GDR GAGC. The first and second named authors are members of the Gruppo Nazionale per le Strutture Algebriche, Geometriche e le loro Applicazioni of the Istituto Nazionale di Alta Matematica "F. Severi" (GNSAGA-INDAM). The mobility of the third named author is suported by a UFI Scholarship VINCI. We thank Antonio Laface and Elisa Postinghel for their helpful comments on a preliminary version of the paper.

\section{Cox rings, Mori dream spaces and log Fano varieties}
Let $X$ be a normal projective $\mathbb{Q}$-factorial variety. We denote by $N^1(X)$ the real vector space of $\mathbb{R}$-Cartier divisors modulo numerical equivalence. 
The \emph{nef cone} of $X$ is the closed convex cone $\Nef(X)\subset N^1(X)$ generated by classes of nef divisors. 

The stable base locus $\textbf{B}(D)$ of a $\mathbb{Q}$-divisor $D$ is the set-theoretic intersection of the base loci of the complete linear systems $|sD|$ for all positive integers $s$ such that $sD$ is integral
\stepcounter{thm}
\begin{equation}\label{sbl}
\textbf{B}(D) = \bigcap_{s > 0}B(sD).
\end{equation}
The \emph{movable cone} of $X$ is the convex cone $\Mov(X)\subset N^1(X)$ generated by classes of 
\emph{movable divisors}. These are Cartier divisors whose stable base locus has codimension at least two in $X$.
The \emph{effective cone} of $X$ is the convex cone $\Eff(X)\subset N^1(X)$ generated by classes of 
\emph{effective divisors}. We have inclusions $\Nef(X)\ \subset \ \overline{\Mov(X)}\ \subset \ \overline{\Eff(X)}$. We refer to \cite[Chapter 1]{De01} for a comprehensive treatment of these topics.

We say that a birational map  $f: X \dasharrow X'$ to a normal projective variety $X'$  is a \emph{birational contraction} if its inverse does not contract any divisor. 
We say that it is a \emph{small $\mathbb{Q}$-factorial modification} 
if $X'$ is $\mathbb{Q}$-factorial and $f$ is an isomorphism in codimension one.
If $f: X \dasharrow X'$ is a small $\mathbb{Q}$-factorial modification then 
the natural pullback map $f^*:N^1(X')\to N^1(X)$ sends $\Mov(X')$ and $\Eff(X')$
isomorphically onto $\Mov(X)$ and $\Eff(X)$ respectively. In particular, we have $f^*(\Nef(X'))\subset \overline{\Mov(X)}$.

Now, assume that the divisor class group $\Cl(X)$ is free and finitely generated, and fix a subgroup $G$ of the group of Weil divisors on $X$ such that the canonical map $G\rightarrow\Cl(X)$, mapping a divisor $D\in G$ to its class $[D]$, is an isomorphism. The \textit{Cox ring} of $X$ is defined as
$$\Cox(X) = \bigoplus_{[D]\in \Cl(X)}H^0(X,\mathcal{O}_X(D))$$
where $D\in G$ represents $[D]\in\Cl(X)$, and the multiplication in $\Cox(X)$ is defined by the standard multiplication of homogeneous sections in the field of rational functions on $X$. 

\begin{Definition}\label{def:MDS} 
A normal projective $\mathbb{Q}$-factorial variety $X$ is called a \emph{Mori dream space}
if the following conditions hold:
\begin{enumerate}
\item[-] $\Pic{(X)}$ is finitely generated, or equivalently $h^1(X,\mathcal{O}_X)=0$,
\item[-] $\Nef{(X)}$ is generated by the classes of finitely many semi-ample divisors,
\item[-] there is a finite collection of small $\mathbb{Q}$-factorial modifications
 $f_i: X \dasharrow X_i$, such that each $X_i$ satisfies the second condition above, and $
 \Mov{(X)} \ = \ \bigcup_i \  f_i^*(\Nef{(X_i)})$.
\end{enumerate}
\end{Definition}

The collection of all faces of all cones $f_i^*(\Nef{(X_i)})$ above forms a fan which is supported on $\Mov(X)$.
If two maximal cones of this fan, say $f_i^*(\Nef{(X_i)})$ and $f_j^*(\Nef{(X_j)})$, meet along a facet,
then there exist a normal projective variety $Y$, a small modification $\varphi:X_i\dasharrow X_j$, and $h_i:X_i\rightarrow Y$, $h_j:X_j\rightarrow Y$ small birational morphisms of relative Picard number one such that $h_j\circ\varphi = h_i$. The fan structure on $\Mov(X)$ can be extended to a fan supported on $\Eff(X)$ as follows. 

\begin{Definition}\label{MCD}
Let $X$ be a Mori dream space.
We describe a fan structure on the effective cone $\Eff(X)$, called the \emph{Mori chamber decomposition}. There are finitely many birational contractions from $X$ to Mori dream spaces, denoted by $g_i:X\dasharrow Y_i$.
The set $\Exc(g_i)$ of exceptional prime divisors of $g_i$ has cardinality $\rho(X/Y_i)=\rho(X)-\rho(Y_i)$.
The maximal cones $\mathcal{C}$ of the Mori chamber decomposition of $\Eff(X)$ are of the form $\mathcal{C}_i \ = \left\langle g_i^*\big(\Nef(Y_i)\big) , \Exc(g_i) \right\rangle$. We call $\mathcal{C}_i$ or its interior $\mathcal{C}_i^{^\circ}$ a \emph{maximal chamber} of $\Eff(X)$. We refer to \cite[Proposition 1.11]{HK00} and \cite[Section 2.2]{Ok16} for details.
\end{Definition}

\begin{Definition}\label{logF}
A normal and $\mathbb{Q}$-factorial projective variety $X$ is \textit{log Fano} if there exists an effective divisor $D\subset X$ such that $-(K_X+D)$ is ample and the pair $(X,D)$ is Kawamata log terminal.
\end{Definition}

By \cite[Corollary 1.3.2]{BCHM10} if $X$ is log Fano then it is a Mori dream space.
  
\subsection{Varieties with a torus action}
Let $X$ be a normal projective variety and $T\times X\rightarrow X$ an effective algebraic torus action on $X$ of complexity one that is such that its biggest $T$-orbits are of codimension one in $X$. 

We will follow the treatment in \cite{HS10}. For a given point $x\in X$, denote by $T_{x}\subset T$ its isotropy group
$$T_{x} =\{t\in T \ | \ t\cdot x=x\}$$
and consider the non-empty $T$-invariant open subset
$$X_{0} =\{x\in X\ | \ \dim(T_{x})=0\}\subset X$$ 
of points in $x$ with zero-dimensional isotropy group. There is a geometric quotient 
$$q:X_{0}\rightarrow X_{0}/T$$ 
where the orbit space $X_{0}/T$ is of dimension one and has a separation that is a rational map $\pi:X_{0}/T\dasharrow \mathbb{P}^{1}$ which is a local isomorphism in codimension one. 

Denote by $E_{1},\dots,E_{m}\subseteq X$ the $T$-invariant prime divisors supported
in $X\setminus X_{0}$ and by $D_{1},\dots,D_{n}\subseteq X$ those $T$-invariant prime divisors who have a finite generic isotropy group of order $l_{j}>1$. Moreover, let $1_{{E}_{k}}$ and $1_{{D}_{j}}$ denote the canonical sections of the divisors $E_{k}$ and $D_{j}$ respectively, and let $1_{q(D_{j})}$ be the canonical section of $q(D_{j})$. 

Choose $r\geq1$ and $a_{0},\dots,a_{r}\in\mathbb{P}^{1}$, such that $\pi$ is an isomorphism
over $\mathbb{P}^{1}\setminus\{a_{0},\dots,a_{r}\}$ and all the divisors $D_{j}$ occur among the 
$$D_{i,j} = q^{-1}(y_{i,j})$$
where
$\pi^{-1}(a_{i}) = \{y_{i,1},\dots,y_{i,{n_i}}\}.$ Let $l_{i,j}\in\mathbb{Z}_{\geq 1}$ denote the order of the generic isotropy group of $D_{i,j}$. \\
For every $0\leq i\leq r$ define a monomial
$$f_{i} = T^{l_{i,1}}_{i,1}\cdots T^{l_{i,n_{i}}}_{i,n_{i}}\in\mathbb{C}[T_{i,j}; 0\leq i\leq r, 1\leq j\leq n_{i}].$$
Moreover, write $a_{i}=[b_{i},c_{i}]$ with $b_{i},c_{i}\in\mathbb{C}$ and for every $0\leq i\leq r-2$ set $k=j+1=i+2$ and define a trinomial
$$g_{i} = (c_{k}b_{j}-c_{j}b_{k})f_{i}+(c_{i}b_{k}-c_{k}b_{i})f_{j}+(c_{j}b_{i}-c_{i}b_{j})f_{k}.$$
In the previous notation we have the following:
\begin{thm}\cite[Theorem 1.3]{HS10}\label{HS} Let $T\times X\rightarrow X$ be an algebraic torus action of complexity one. Then, in terms of the data defined above, the Cox ring of $X$ is given as
	$$\Cox(X)=\mathbb{C}[S_{1},\dots,S_{m},T_{i,j};\: 0\leq i\leq r, 1\leq j\leq n_{i}] /\langle g_{i}; 0\leq i\leq r-2\rangle$$
	where $1_{{E}_{k}}$ corresponds to $S_{k}$ and $1_{{D}_{i,j}}$ to $T_{i,j}$, and the $\text{Cl}(X)$-grading on the right hand side is defined by associating to $S_{k}$ the class of $E_{k}$ and to $T_{i,j}$ the class of $D_{i,j}$. In particular, $\Cox(X)$ is finitely generated.
\end{thm} 
Theorem \ref{HS} will be crucial in the study of the Cox ring of $X^{1,n}_{n+1}$.

\section{Mori cones}\label{sectionmori}
In this section we compute the cone of effective curves of $X^{1,n}_r$ for $r\leq n+2$, and for $r = n+3$ when $n\leq 4$.

\begin{Proposition}\label{Mori_n+1}
The Mori cone of $X^{1,n}_r$ is given by
$$
\NE(X^{1,n}_r)=\langle h_{1}-e_{i},h_{2}-e_{i},e_{i}\rangle
$$
for all $r\leq n+1$.
\end{Proposition}
\begin{proof}
It is enough to prove the claim for $r = n+1$. The case $n = 1$ is well-known. We will assume that the claim is proved for $X^{1,n-1}_n$ and prove it for $X^{1,n}_{n+1}$. 

Let $C\subset X^{1,n}_{n+1}$ be an irreducible curve. If $C$ is contracted by $\pi$ then $C\subset E_i$ for some $i = 1,\dots, n+1$ and hence $C$ is a multiple of $e_i$. Assume that $\pi(C)$ is a curve of bidegree $(d_1,d_2)$ and such that $\mult_{p_i}C = m_i$. Then we may write 
$$
C\sim d_{1}h_{1}+d_{2}h_{2}-m_{1}e_{1}-\dots-m_{n}e_{n}-m_{n+1}e_{n+1}.
$$
Note that $C\cdot (H_1-E_i) = d_1-m_i$. If $d_1-m_i < 0$ then $C$ is contained in the strict transform of the fiber of $\pi_1$ passing through $p_i$. Such fiber is $\mathbb{P}^n$ blown-up in a point and hence $C$ may be written as a linear combination with non negative coefficients of $e_i$ and $h_2-e_i$. Hence, we may assume that $d_1-m_i \geq 0$ for all $i = 1,\dots,n+1$. 

Now, consider the projection $\widetilde{\pi}_2(C)$. If $\widetilde{\pi}_2(C)$ is a point then $C$ is a linear combination with non negative coefficients of $h_1-e_i$ and $e_i$ for some $i = 1,\dots,n+1$. Assume that $\widetilde{\pi}_2(C)$ is a curve. Then this curve has degree $d_2$ and $\mult_{\pi_2(p_i)}\widetilde{\pi}_2(C) = m_i$. Let $\Pi_{1,\dots,n}\subset\mathbb{P}^n$ be the hyperplane generated by $p_1,\dots,p_n$. If
$$
m_1 + \dots + m_n > d_2
$$
then $\widetilde{\pi}_2(C)\subset\Pi_{1,\dots,n}$ and hence $C$ is contained in the strict transform of $\mathbb{P}^1\times \Pi_{1,\dots,n}$ which is isomorphic to $X^{1,n-1}_n$. Therefore, we can conclude by induction on $n$. On the other hand, if 
$$
m_1 + \dots + m_n \leq d_2
$$
we may write
$$
C\sim \sum_{i=1}^n m_i(h_2-e_i) + m_{n+1}(h_{1}-e_{n+1}) + (d_{1}-m_{n+1})h_{1} + (d_{2}-m_{1}-\dots-m_{n})h_{2}
$$
where all the coefficients are non negative, concluding the proof.
\end{proof}

Consider the Segre embedding
$$
\begin{array}{cccc}
\sigma_{1,n}: & \mathbb{P}^1\times\mathbb{P}^n & \longrightarrow & \mathbb{P}^{2n+1} \\ 
 & ([u_0,u_1],[v_0,\dots,v_n]) & \mapsto & [u_0v_0,u_0v_1,\dots,u_1v_n]
\end{array} 
$$
and set $\Sigma^{1,n} = \sigma_{1,n}(\mathbb{P}^1\times\mathbb{P}^n)$. Recall that $\deg(\Sigma^{1,n}) = n+1$.

\begin{Proposition}\label{Mori_n+2}
The Mori cone of $X^{1,n}_{n+2}$ is given by
$$
\NE(X^{1,n}_{n+2})=\langle h_{1}-e_{i},h_{2}-e_{i},e_{i},h_{1}+nh_{2}-e_{1}-\cdots-e_{n+2}\rangle.
$$                               
\end{Proposition} 
\begin{proof}
The case $n = 1$ is well-known. Let $C$ be an irreducible curve in $X^{1,n}_{n+2}$. If $C$ is contracted by $\pi$ then $C\subset E_i$ for some $i = 1,\dots, n+2$ and hence $C$ is a multiple of $e_i$. Assume that $\pi(C)$ is a curve of bidegree $(d_1,d_2)$ and such that $\mult_{p_i}C = m_i$. Then we may write 
$$
C\sim d_{1}h_{1}+d_{2}h_{2}-m_{1}e_{1}-\dots-m_{n+1}e_{n+1}-m_{n+2}e_{n+2}.$$
If
$$ d_{1}+d_{2}\geq m_{1}+\dots+m_{n+2}$$ 
then $C$ may be written as a linear combination with non negative coefficients of $e_i$, $h_2-e_i$ and $h_{1}-e_{i}$. Recall that $\sigma_{1,n}(\pi(C))$ is a curve of degree $d_{1}+d_{2}$ passing through the points $\sigma_{1,n}(p_{1}),\dots,\sigma_{1,n}(p_{n+2})$ with multiplicities $m_{i}$. 

Let us denote by $\Pi_{1,\dots,n+2}\subset\mathbb{P}^{2n+1}$ the $(n+1)$-plane spanned by $\sigma_{1,n}(p_1),\dots,\sigma_{1,n}(p_{n+2})$. 
If 
$$d_{1}+d_{2}< m_{1}+\dots+m_{n+2}$$ 
then $\sigma_{1,n}(\pi(C))\subset \Pi_{1,\dots,n+2}$. Since $\Sigma^{1,n}$ has degree $n+1$, the intersection $\Pi_{1,\dots,n+2} \cap \Sigma^{1,n}$ is a rational normal curve of degree $n+1$ and with class $h_{1}+nh_{2}-e_{1}-\dots-e_{n+2}$.	
\end{proof}

Let $a,b\geq 1$ be two integers, fix two complementary subspace $\Lambda^a,\Lambda^b\subset\mathbb{P}^{a+b+1}$, two rational normal curves $C_a\subset\Lambda^a, C_b\subset\Lambda^b$, and an isomorphism $\phi:C_a\rightarrow C_b$. The surface 
$$
S_{(a,b)} = \bigcup_{x\in C_a}\left\langle x,\phi(x)\right\rangle\subset\mathbb{P}^{2n+1}
$$
where $\left\langle x,\phi(x)\right\rangle$ denotes the line through $x,\phi(x)$, is a rational normal scroll of type $(a,b)$. This is a smooth rational surface of degree $\deg(S_{(a,b)}) = a+b$.

\begin{Lemma}\label{sur}
Let $H\subset\mathbb{P}^{2n+1}$ be a general $(n+2)$-plane. Then the intersection $H\cap \Sigma^{1,n}$ is a rational normal scroll $S_{(a,b)}$ with
$$
(a,b) = \left\lbrace
\begin{array}{ll}
(\frac{n+1}{2},\frac{n+1}{2}) & \text{if } n \text{ is odd};\\ 
(\frac{n}{2},\frac{n+2}{2}) & \text{if } n \text{ is even}.
\end{array}\right.
$$
\end{Lemma}
\begin{proof}
The Segre variety $\Sigma^{1,n}$ is the projectivization over $\mathbb{P}^1$ of the rank $n+1$ vector bundle $\mathcal{O}_{\mathbb{P}^1}(-1)^{n+1}$. A codimension $n-1$ linear section corresponds to the projectivization of the kernel of a morphism
$$
\mathcal{O}_{\mathbb{P}^1}(-1)^{n+1}\rightarrow\mathcal{O}_{\mathbb{P}^1}^{n-1}
$$
which is a rank two vector bundle $\mathcal{O}_{\mathbb{P}^1}(-a)\oplus\mathcal{O}_{\mathbb{P}^1}(-b)$. To conclude it is enough to note that for $H$ general the splitting type $(-a,-b)$ is the one given in the statement. 
\end{proof}

\begin{Lemma}\label{lem2}
Let $f:X\rightarrow Y$ be a morphism of projective varieties, and $\NE(f)$ the cone of curves contracted by $f$. Then $\NE(f)$ is extremal in $\NE(X)$.  
\end{Lemma}
\begin{proof}
Let $\Gamma$ be the class of an irreducible curve in $\NE(f)$, and assume that $\Gamma = \Gamma_1 + \Gamma_2$ for $\Gamma_1,\Gamma_2\in\NE(X)$. Applying $f_{*}$ we get $f_{*}\Gamma_1 + f_{*}\Gamma_2 = f_{*}\Gamma = 0$ since $\Gamma$ is contracted by $f$. Therefore, $f_{*}\Gamma_1 = f_{*}\Gamma_2 = 0$ and hence $\Gamma_1,\Gamma_2 \in \NE(f)$.
\end{proof}

\begin{Proposition}\label{nfg_n+3}
Fix $p_1,\dots,p_{n+3}\in \Sigma^{1,n}$ general points. Set $H = \left\langle p_1,\dots,p_{n+3}\right\rangle$. Let $\widetilde{S}_{a,b}$ be the blow-ups of $S_{a,b} = H\cap\Sigma^{1,n}$ at the $p_i$. Then 
$$
\NE(X^{1,n}_{n+3}) = \left\langle \NE(\widetilde{S}_{a,b}), h_1-e_1,\dots,h_1-e_{n+3}\right\rangle.
$$
\end{Proposition}
\begin{proof}
Let $C\sim ah_1 + bh_2 - \sum_{i=1}^{n+3}m_ie_i$ be an irreducible curve in $X_{n+3}^{1,n}$, and $\Gamma\subset\Sigma^{1,n}$ its image in $\mathbb{P}^{2n+1}$. Then $\deg(\Gamma) = a+b$ and $\mult_{p_i}\Gamma = m_i$ for $i = 1,\dots,n+3$. 

If $a+b < \sum_{i=1}^{n+3}m_i$ then $\Gamma$ is contained in all the hyperplanes containing $H$ and hence $\Gamma\subset H\cap \Sigma^{1,n}$. By Lemma \ref{sur} $S_{a,b} = H\cap \Sigma^{1,n}$ is a scroll.

If $a+b\geq \sum_{i=1}^{n+3}m_i$ we can pair each $e_i$ with one among $h_1$ and $h_2$, and write $C$ as a linear combination with non negative coefficients of $h_1-e_i,h_2-e_j,e_k$.  

The curves of class $h_2-e_i$ are numerically equivalent to the strict transform of the line through $p_i$ in the ruling of $S_{a,b}$. The curves of class $h_1-e_i$ are contracted by $\widetilde{\pi}_2$ and the curves of class $h_2-e_j$ are contracted by $\widetilde{\pi}_1$. Hence, by Lemma \ref{lem2} $h_1-e_i$ and $h_2-e_j$ generate extremal rays of $\NE(X_{n+3}^{1,n})$.

Summing-up, we have  showed that an irreducible curve $C\subset X_{n+3}^{1,n}$ can be written as a linear combination with non negative coefficients of a curve $\Gamma\subset\widetilde{S}_{a,b}$ and of the $h_1-e_i$.
\end{proof}

\subsection{Mori cones of del Pezzo surfaces}\label{dP}
Let $S_r$ be the blow-up of $\mathbb{P}^2$ at $p_1,\dots,p_r\in\mathbb{P}^2$ general points. The surface $S_r$ is del Pezzo if and only if $0\leq r\leq 8$. We recall the structure of the Mori cone $\NE(S_r)$ for $r = 6,7,8$. We will denote by $\overline{h}$ the pull-back of a line in $\mathbb{P}^{2}$ and by $\overline{e}_{i}$ the exceptional divisor over the point $p_{i}$ for $i=1,\dots,r$. 

\subsubsection{del Pezzo of degree $3$}\label{dP3}
The Mori cone of $S_3$ is generated by the following classes
$$
\begin{tabular}{l|c}
\textit{Divisor class} & \textit{Number of extremal rays} \\ 
\hline 
$\overline{e}_i$ & 6 \\ 
$\overline{h}-\overline{e}_i-\overline{e}_j$ & 15 \\
$2\overline{h}-\overline{e}_{i_1}-\dots-\overline{e}_{i_5}$ & 6
\end{tabular} 
$$
for a total of $27$ extremal rays.

\subsubsection{del Pezzo of degree $2$}\label{dP2}
The Mori cone of $S_2$ is generated by the following classes
$$
\begin{tabular}{l|c}
\textit{Divisor class} & \textit{Number of extremal rays} \\ 
\hline 
$\overline{e}_i$ & 7 \\ 
$\overline{h}-\overline{e}_i-\overline{e}_j$ & 21\\
$2\overline{h}-\overline{e}_{i_1}-\dots-\overline{e}_{i_5}$ & 21\\
$3\overline{h}-2\overline{e}_{i}-\overline{e}_{j_1}-\dots-\overline{e}_{j_6}$ & 7
\end{tabular}
$$
for a total of $56$ extremal rays.

\subsubsection{del Pezzo of degree $1$}\label{dP1}
The Mori cone of $S_1$ is generated by the following classes
$$
\begin{tabular}{l|c}
\textit{Divisor class} & \textit{Number of extremal rays} \\ 
\hline 
$\overline{e}_i$ & 8 \\ 
$\overline{h}-\overline{e}_i-\overline{e}_j$ & 28\\
$2\overline{h}-\overline{e}_{i_1}-\dots-\overline{e}_{i_5}$ & 56\\
$3\overline{h}-2\overline{e}_{i}-\overline{e}_{j_1}-\dots-\overline{e}_{j_6}$ & 56\\
$4\overline{h}-2\overline{e}_{i_1}-\dots-2\overline{e}_{i_3}-\overline{e}_{j_1}-\dots-\overline{e}_{j_5}$ & 56\\
$5\overline{h}-2\overline{e}_{i_1}-\dots -2\overline{e}_{i_6}-\overline{e}_{j_1}-\overline{e}_{j_2}$ & 28\\
$6\overline{h}-3\overline{e}_{i}-2\overline{e}_{i_1}-\dots-2\overline{e}_{j_7}$ & 8
\end{tabular}
$$
for a total of $240$ extremal rays.

\begin{Proposition}\label{Mori_n+3}
	If $n\leq 4$, the Mori cone of $X^{1,n}_{n+3}$ is given by
	$$
	\NE(X^{1,n}_{n+3})=\langle h_{1}-e_{i},h_{2}-e_{i},e_{i},h_{1}+nh_{2}-e_{i_{1}}-\dots-e_{i_{n+2}}\rangle
	$$ 
	for $i,i_{1},\dots,i_{n+2}\in\{1,\dots,n+3\}$.
\end{Proposition}
\begin{proof}
	The case $n = 1$ is well-known.
	Let $C$ be an irreducible curve in $X^{1,n}_{n+3}$. If $C$ is contracted by $\pi$ then $C\subset E_i$ for some $i = 1,\dots, n+3$ and hence $C$ is a multiple of $e_i$. Now, assume that $\pi(C)$ is a curve of bidegree $(d_1,d_2)$ in $\P^1\times \P^n$ and such that $\mult_{p_i}C = m_i$. This means that we can write 
$$C\sim d_{1}h_{1}+d_{2}h_{2}-m_{1}e_{1}-\dots-m_{n+2}e_{n+2}-m_{n+3}e_{n+3}.$$
First, we suppose that $$d_{1}+d_{2}\geq m_{1}+\dots+m_{n+3}.$$ Then $C$ may be written as a linear combination with non negative coefficients of $e_i$, $h_2-e_i$ and $h_{1}-e_{i}$. 
	
Note that $\sigma_{1,n}(\pi(C))$ is a curve of degree $d_{1}+d_{2}$ passing through the points $\sigma_{1,n}(p_{1}),\dots,\sigma_{1,n}(p_{n+3})$ with multiplicities $m_{i}$.	Let $\Pi_{1,\dots,n+3}\subset\mathbb{P}^{2n+1}$ be the $(n+2)$-plane generated by $\sigma_{1,n}(p_1),\dots,\sigma_{1,n}(p_{n+3})$. If
$$d_{1}+d_{2}< m_{1}+\dots+m_{n+3}$$ 
then $\sigma_{1,n}(\pi(C))\subset \Pi_{1,\dots,n+3}$. By Lemma \ref{sur} the intersection $X = \Pi_{1,\dots,n+3}\cap \Sigma^{1,n}$ is a degree $n+1$ scroll in $\mathbb{P}^{2n+1}$, and by Proposition \ref{nfg_n+3} the Mori cone of $X^{1,n}_{n+3}$ is generated by the Mori cone of $X$ and the $h_1-e_i$ for $i=1,\dots,n+3$. Note that $\sigma_{1,n}(\pi(C))\subset X$. Let us consider separately the cases $n=2,3,4$.
	
If $n=2$ then $X$ is a cubic scroll and the projection $\pi_2:X \to \mathbb{P}^{2}$ is the blow-down of the exceptional divisor. Hence, $C\subset \Bl_{5}X\cong\Bl_{6}\mathbb{P}^{2}$, which is a del Pezzo surface of degree three. Let $\overline{h}$ be the pull-back of a line in $\P^2$ and $\overline{e}_i$ the classes of the exceptional divisors. The Mori cone of $\Bl_{6}\mathbb{P}^{2}$ is described in \ref{dP3}. The isomorphism $\Bl_{6}\mathbb{P}^{2}\rightarrow \Bl_{5}X\subset X_{5}^{1,2}$ induces a map $N_1(\Bl_{6}(\mathbb{P}^{2})) \to N_1(X_{5}^{1,2})$ such that
$$
\overline{e}_{1}\mapsto h_{1},\: \overline{h} \mapsto h_1+ h_2,\: \overline{e}_{i}\mapsto e_{i-1},\: \text{for}\: i = 2,\dots,5. 
$$
To conclude it is enough to show that the generators of $\NE(\Bl_{6}\mathbb{P}^{2})$ can be written as linear combinations with non negative coefficients of the effective classes listed in the statement. 
	
First, note that $\overline{e}_i = e_{i-1}$ for $i = 2,\dots,6$, and $\overline{e}_1 = h_1 = (h_1 - e_i) + e_i$. Furthermore
\begin{scriptsize}
$$
\begin{array}{ll}
\overline{h}-\overline{e}_i-\overline{e}_j & = (h_1-e_{i-1})+(h_2-e_{j-1}) \text{ for } i,j = 2,\dots,6;\\
\overline{h}-\overline{e}_1-\overline{e}_j & = h_2-e_{j-1} \text{ for } j = 2,\dots,6;\\
2\overline{h}-\overline{e}_{i_1}- \dots - \overline{e}_{i_5} & = (h_1+2h_2-{e}_{i_1-1}- \dots - {e}_{i_4-1})+(h_1-{e}_{i_5-1}) \text{ for } 2\leq i_1,\dots, i_5 \leq 6;\\
2\overline{h}-\overline{e}_{1}- \overline{e}_{j_1} - \dots - \overline{e}_{j_4} & = h_1 + 2h_2 - -e_{j_1-1} - \dots - e_{j_4-1}.
\end{array}
$$
\end{scriptsize}
If $n=3$ then $X$ is a quartic scroll given by the intersection in $\mathbb{P}^1\times\mathbb{P}^3$ of two hypersurfaces of bidegree $(1,1)$
$$X = \{x_{0}f_{1}+x_{1}f_{2} = x_{0}g_{1}+x_{1}g_{2}=0\}\subset \mathbb{P}^1\times\mathbb{P}^3$$
with $f_i,g_i \in k[y_0,\dots,y_3]_1$.

The image of the projection $X\rightarrow\mathbb{P}^3$ is the quadric surface $\{f_1g_2-f_2g_1 = 0\}\subset\mathbb{P}^3$. Note that since $f_1,f_2,g_1,g_2$ do not vanish simultaneously at a point such projection is an isomorphism. 	
	
Hence, we have that	$C\subset \Bl_{6}X\cong\Bl_{6}(\mathbb{P}^{1}\times\mathbb{P}^{1})\cong \Bl_{7}(\mathbb{P}^{2}),$ which is a del Pezzo surface of degree two. The Mori cone of $\Bl_{7}(\P^2)$ is described in \ref{dP2}. The inclusion $\Bl_{7}(\mathbb{P}^{2})\rightarrow X^{1,3}_{6}$ maps
$$
\overline{h}\mapsto h_1+2h_2-e_1,\: \overline{e}_1\mapsto h_2-e_1,\: \overline{e}_2\mapsto h_1+h_2-e_1, \: \overline{e}_i\mapsto e_{i-1},\: \text{for}\: i= 3,\dots 7.
$$
For the $\overline{e}_i$ there is nothing to prove. We have 
\begin{scriptsize}
$$
\begin{array}{ll}
\overline{h}-\overline{e}_1-\overline{e}_2 & = e_1;\\
\overline{h}-\overline{e}_1-\overline{e}_i & = (h_1-e_1)+(h_2-e_{i-1})+e_1;\\
\overline{h}-\overline{e}_2-\overline{e}_i & = h_2-e_{i-1};\\
\overline{h}-\overline{e}_i-\overline{e}_j & = (h_1-e_1)+(h_2-e_{i-1})+(h_2-e_{j-1});\\
2\overline{h}-\overline{e}_1-\overline{e}_2 -\overline{e}_{i_1}-\overline{e}_{i_2}-\overline{e}_{i_3} & = (h_1 -e_{i_1-1})+(h_2 -e_{i_2-1}) + (h_2 -e_{i_3-1});\\
2\overline{h}-\overline{e}_1 -\overline{e}_{i_1}-\overline{e}_{i_2}-\overline{e}_{i_3}-\overline{e}_{i_4} & = (h_1-e_1)+(h_1-{e}_{i_1-1})+(h_2-{e}_{i_2-1})+(h_2-{e}_{i_3-1})+(h_2-{e}_{i_4-1});\\
2\overline{h}-\overline{e}_2 -\overline{e}_{i_1}-\overline{e}_{i_2}-\overline{e}_{i_3}-\overline{e}_{i_4} & =
h_1+3h_2-e_1 - e_{i_1-1} - e_{i_2-1} - e_{i_3-1} - e_{i_4-1};\\
3\overline{h}-2\overline{e}_1-\overline{e}_2-\overline{e}_{i_1}-\dots-\overline{e}_{i_5} & = h_1+(h_1+3h_2-e_{i_1-1}-\dots-e_{i_5-1});\\ 
3\overline{h}-\overline{e}_1-2\overline{e}_2-\overline{e}_{i_1}-\dots-\overline{e}_{i_5} & =  h_1+3h_2-e_{i_1-1}-\dots-e_{i_5-1};\\
3\overline{h}-\overline{e}_1-\overline{e}_2-2\overline{e}_{i_1}-\dots-\overline{e}_{i_5} & =  (h_1-e_1)+(h_2-e_{i_1-1})+(h_1+3h_2-e_{i_1-1}-\dots-e_{i_5-1}).
\end{array}
$$
\end{scriptsize}
If $n=4$ then $X$ is a complete intersection in $\mathbb{P}^{1}\times\mathbb{P}^{4}$ of three hypersurfaces of bidegree $(1,1)$
$$X = \{x_{0}f_{1}+x_{1}f_{2} = x_{0}g_{1}+x_{1}g_{2}= x_{0}h_{1}+x_{1}h_{2} = 0\}\subset \mathbb{P}^{1}\times\mathbb{P}^{4}$$
with $f_{i},g_{i},h_i\in k[y_{0},\dots,y_{4}]_{1}$. Note that $X$ is projected onto a cubic $3$-fold $S\subset\mathbb{P}^4$. Recall that $S$ is isomorphic to the blow-up of $\mathbb{P}^2$ at a point that we can assume to be $[1:0:0]$. The rational map 
$$
\begin{array}{cccc}
\tau: & \mathbb{P}^2 & \dasharrow & \mathbb{P}^1\times\mathbb{P}^4\\
 & [x,y,z] & \longmapsto & ([y,z],[xy,xz,y^2,yz,z^2])
\end{array}
$$
yields an embedding $\widetilde{\tau}:S\rightarrow \mathbb{P}^1\times\mathbb{P}^4$. Denote by $\overline{h}$ the pull-back to $S$ of a line in $\mathbb{P}^2$ and by $\overline{e}_1\subset S$ the exceptional divisor over $[1:0:0]$. Hence, $X$ is isomorphic to $\mathbb{P}^2$ blown-up at eight points that is to a del Pezzo surface of degree one whose generators of the Mori cone are listed in \ref{dP1}. Then $\widetilde{\tau}$ lifts to an embedding $X\rightarrow X^{1,4}_7$ which in turn maps
$$
\overline{h}\mapsto h_1+2h_2,\: \overline{e}_1\mapsto h_1+h_2,\: \overline{e}_{i}\mapsto e_{i-1}, \: \text{for}\: i = 2,\dots,8.
$$
We have
\begin{scriptsize}
$$
\begin{array}{ll}
\overline{h}-\overline{e}_1-\overline{e}_j & =  h_2 -e_{j-1}; \\
\overline{h}-\overline{e}_i-\overline{e}_j & =  (h_1-e_1)+(h_2-e_{i-1})+(h_2-e_{j-1}) + e_1; \\
2\overline{h}-\overline{e}_{1}-\overline{e}_{i_1}-\dots-\overline{e}_{i_4} & = (h_1 - e_{i_1-1}) + (h_2 - e_{i_2-1}) + (h_2 - e_{i_3-1}) + (h_2 - e_{i_4-1});  \\
2\overline{h}-\overline{e}_{i_1}- \dots-\overline{e}_{i_5} & = \sum_{k=1}^2 (h_1 - e_{i_k-1})+ \sum_{t=3}^5 (h_2 - e_{i_t-1}) + (h_2 - e_{1}) + e_1;\\
3\overline{h}-2\overline{e}_{i}-\overline{e}_{j_1}-\dots-\overline{e}_{j_6} & = 2(h_1 - e_{i-1}) + (h_1- e_{j_1-1}) + \sum_{k=1}^6(h_2 - e_{j_k-1}) +  e_{j_1-1}; \\
3\overline{h}-2\overline{e}_{1}-\overline{e}_{j_1}-\dots-\overline{e}_{j_6} & = 
h_1+4h_2 -e_{j_1-1}-\dots - e_{j_6-1};\\
3\overline{h}-2\overline{e}_{i}-\overline{e}_1 - \overline{e}_{j_1}-\dots-\overline{e}_{j_5} & = 2(h_1 - e_{i-1}) + (h_2 - e_{j_1-1}) + \dots + (h_2 - e_{j_5-1}); \\
4\overline{h}-2\overline{e}_{1}-2\overline{e}_{i_1}-2\overline{e}_{i_2}-\overline{e}_{j_1}-\dots-\overline{e}_{j_5} & =  2(h_2 -e_{i_1-1})+(h_1 - e_{i_2-1})+(h_1+ 4h_2 -e_{i_2-1} - \sum_{k=1}^5e_{j_k-1});  \\
4\overline{h}-2\overline{e}_{i_1}-\dots-2\overline{e}_{i_3}-\overline{e_1}-\overline{e}_{j_1}-\dots-\overline{e}_{j_4} & =  \sum_{k=1}^3 2(h_2 - e_{i_k}) + (h_2 - e_{j_1-1}) + \sum_{t=2}^4(h_1 - e_{j_t-1}); \\
5\overline{h}-2\overline{e}_1-2\overline{e}_{i_1}-\dots -2\overline{e}_{i_5}-\overline{e}_{j_1}-\overline{e}_{j_2} & = (h_1 + 4h_2 -2\sum_{k=1}^3e_{i_k-1})+2\sum_{t=1}^2(h_t-e_{i_{t+3}-1})+\sum_{k=1}^2(h_2-e_{j_k-1}); \\
5\overline{h}-2\overline{e}_{i_1}-\dots -2\overline{e}_{i_6}-\overline{e}_{1}-\overline{e}_{j} & = \sum_{k=1}^4 2(h_2 - e_{i_k-1}) + \sum_{t=5}^6 2(h_1- e_{i_t-1}) + (h_2 - e_{j-1}); \\
6\overline{h}-2\overline{e}_1-3\overline{e}_2-2\overline{e}_{3}-\dots-2\overline{e}_{8}  & = 2(h_1+4h_2-e_3-\dots -e_8) + (h_1-e_2) + 2(h_2-e_2);\\
6\overline{h}-3\overline{e}_1-2\overline{e}_2-\dots-2\overline{e}_{8} & = 2(h_1+4h_2-e_2-\dots -e_7) + (h_1-e_8) + (h_2-e_8);
\end{array}
$$
\end{scriptsize}
concluding the proof.
\end{proof}

\begin{Corollary}\label{Nefcn+1}
If $r\leq n+1$ the nef cone of $X^{1,n}_{r}$ is given by 
$$
\Nef(X^{1,n}_{r}) = \left\langle H_1,H_2,H_1+H_2-E_{i_1}-\dots - E_{i_r}\right\rangle
$$
with $r = 1,\dots,n+1$. In particular, $\Nef(X^{1,n}_{n+1})$ has $2^{n+1}+1$ extremal rays.
\end{Corollary}
\begin{proof}
The claim follows by duality from Proposition \ref{Mori_n+1}. 
\end{proof}

\begin{Corollary}\label{Nefcn+2}
The nef cone of $X^{1,n}_{n+2}$ is given by 
$$
\Nef(X^{1,n}_{n+2}) = \left\langle H_1,H_2,H_1+H_2-E_{i_1}-\dots - E_{i_r},2H_1+H_2-\sum_{i=1}^{n+2}E_i,nH_1+(n+1)H_2-n\sum_{i=1}^{n+2}E_i\right\rangle
$$
with $r = 1,\dots,n+1$. In particular, $\Nef(X^{1,n}_{n+2})$ has $2^{n+2}+2$ extremal rays.
\end{Corollary}
\begin{proof}
The claim follows by duality from Proposition \ref{Mori_n+2}. 
\end{proof}

\begin{Corollary}\label{Nefcn+3}
Consider the following divisor classes on $X_{n+3}^{1,n}$:
$$
\begin{array}{lll}
D_t & = & H_{1}+H_2-E_{i_1}-\dots -E_{i_t};\\ 
D_r & = & 2H_1+H_{2}-E_{i_{1}}-\dots-E_{i_{r}};\\
D_r' & = & nH_1+(n+1)H_{2}-nE_{i_{1}}-\dots-nE_{i_{r}};\\
D_{k,s} & = & kH_{1}+kH_{2}-kE_{i_1}-\dots-kE_{i_2}-(k-1)E_{i_{s+1}}-\dots -(k-1)E_{i_{n+3}}.
\end{array} 
$$
If $n\leq 4$ then the nef cone of $X^{1,n}_{n+3}$ is given by 
$$
\Nef(X^{1,n}_{n+3}) = \left\langle H_1,H_2,D_t,D_r,D_r',D_{k,s}\right\rangle
$$
with $t = 1,\dots,n+1$, $r\in\{n+2,n+3\}$, $k = 2,\dots,n+1$ and $s = n-k+2$. In particular, if $n\leq 4$ then $\Nef(X^{1,n}_{n+3})$ has $2^{n+4}-\frac{(n+3)(n+2)}{2}$ extremal rays.
\end{Corollary}
\begin{proof}
The claim follows by duality from Proposition \ref{Mori_n+3}. 
\end{proof}

\section{The Cox ring of $X^{1,n}_{n+1}$}\label{SCox_n+1}
The $n$-dimensional complex torus $T=(\mathbb{C}^{*})^{n}$ acts on $\mathbb{P}^{1}\times\mathbb{P}^{n}$ as follows:
\stepcounter{thm}
\begin{equation}\label{action}
\begin{array}{ccc}
T\times(\mathbb{P}^{1}\times\mathbb{P}^{n})& \longrightarrow & \mathbb{P}^{1}\times\mathbb{P}^{n}\\
((t_{1},\dots,t_{n}),([x_{0},x_{1}],[y_{0},\dots,y_{n}])) & \longmapsto & ([x_{0},x_{1}],[y_{0},t_{1}y_{1},t_{2}y_{2},\dots,t_{n}y_{n}]).
\end{array}
\end{equation}
Let us write the $n+1$ general points $p_1,\dots,p_{n+1}\in \mathbb{P}^{1}\times\mathbb{P}^{n}$ as follows:
$$
\begin{array}{l}
p_{1} = ([\alpha_{1},\beta_{1}],[1,0,\dots,0]);\\
p_{2} = ([\alpha_{2},\beta_{2}],[0,1,0,\dots,0]);\\
p_{3} = ([\alpha_{3},\beta_{3}],[0,0,1,0,\dots,0]);\\
p_{4} = ([\alpha_{4},\beta_{4}],[0,0,0,1,0,\dots,0]);\\
\vdots\\
p_{n+1} = ([\alpha_{n+1},\beta_{n+1}],[0,\dots,0,1]);
\end{array}
$$
where $[\alpha_{1},\beta_{1}] = [0,1]$, $[\alpha_{2},\beta_{2}] = [1,0]$ and $[\alpha_{3},\beta_{3}] = [1,1]$. Note that (\ref{action}) lifts to an action of $T$ on $X^{1,n}_{n+1}$ whose orbits of maximal dimension have codimension one in $X^{1,n}_{n+1}$. Hence, $X^{1,n}_{n+1}$ has complexity one and Theorem \ref{HS} implies that $X^{1,n}_{n+1}$ is a Mori dream space. We begin by proving the following stronger statement:

\begin{Proposition}\label{logf}
The variety $X_{r}^{1,n}$ is log Fano for $r\leq n+1$.
\end{Proposition}
\begin{proof}
Consider the case $r = n+1$. The result for $r < n+1$ will then follow from \cite[Theorem 2.9]{PS09}. The claim amounts to find an effective $\mathbb{Q}$-divisor $\Delta$ on $X_{n+1}^{1,n}$ such that $-K_{X_{n+1}^{1,n}}-\Delta$ is ample, and the pair $(X_{n+1}^{1,n}, \Delta)$ is Kawamata log terminal. Recall that 
$$    
-K_{X_{n+1}^{1,n}} = 2H_1 + (n+1)H_2 - nE_1 - \dots - nE_{n+1}
$$
and consider the divisor
$$D =(n+1)H_2 -nE_1 -\dots -nE_{n+1}.$$
The divisor $D$ is effective since it is the pull-back of the effective divisor on $\P^n$ consisting in the union of the $(n+1)$ hyperplanes passing through $n$ among the $n+1$ projections of the blown-up points. Set $\Delta_\epsilon = \epsilon D$ with $\epsilon\in\mathbb{Q}_{>0}$, and consider the divisor
$$-K_{X_{n+1}^{1,n}}-\Delta_\epsilon = 2H_1 + ((n+1)-\epsilon (n+1))H_2 + \sum_{i=1}^{n+1}(\epsilon n - n)E_i.$$
By Proposition \ref{Mori_n+1} we have that $\NE(X_{n+1}^{1,n})=\langle h_{1}-e_{i},h_{2}-e_{i},e_{i}\rangle$. Since
$$
\begin{array}{lll}
(-K_{X_{n+1}^{1,n}}-\Delta_\epsilon)\cdot e_{i}=n-\epsilon n>0 & \text{if and only if} &  \epsilon<1;\\
(-K_{X_{n+1}^{1,n}}-\Delta_\epsilon)\cdot(h_{1}-e_{i})=2-n+\epsilon n>0 & \text{if and only if} & \epsilon>(n-2)/n;\\
(-K_{X_{n+1}^{1,n}}-\Delta_\epsilon)\cdot(h_{2}-e_{i})=n+1-\epsilon(n+1)-n+\epsilon n>0 & \text{if and only if} & \epsilon<1;\\
\end{array}
$$
and hence the divisor $-K_{X_{n+1}^{1,n}}-\Delta_\epsilon$ is ample if and only if $(n-2)/n<\epsilon<1$.

Next, we show that the pair $(X_{n+1}^{1,n}, \Delta_\epsilon)$ is Kawamata log terminal. Consider the images $q_i =\pi_2(p_i)$ of the blown-up points via the second projection. There exist $n+1$ hyperplanes in $\P^n$ through each subset of $n$ points among the $q_i$, and through each $q_i$ there pass $n$ of them. Let us denote by $\widetilde{H}_i\subset X_{n+1}^{1,n}$ the strict transforms of the inverse images via the second projection of these hyperplanes. The divisor $D$ has multiplicity $n$ along the strict transforms of the curves $\pi_2^{-1}(q_i)$ for all $i$. Hence, any $n$ among the divisors $\widetilde{H}_i$ intersect transversally along one of the strict transforms of these curves. 

Generalizing this argument, fix a set of $m$ points among the $q_i$, with $1\leq m \leq (n-1)$, and denote by $\Lambda_m$ their linear span. The divisor $D$ has multiplicity $n+1-m$ along the strict transform $\widetilde{\Lambda}_m$ of $\pi_2^{-1}(\Lambda_m)$. Any $n+1-m$ among the $\widetilde{H}_i$ intersect transversally along a subvariety of type $\widetilde{\Lambda}_m\subset X_{n+1}^{1,n}$, and hence $\Delta_\epsilon$ is a simple normal crossing effective divisor for all $\epsilon > 0$. Summing-up, we proved that for all $(n-2)/n<\epsilon<1$ the divisor $-K_{X_{n+1}^{1,n}}-\Delta_\epsilon$ is ample and the pair $(X_{n+1}^{1,n},\Delta_\epsilon)$ is Kawamata log terminal concluding the proof.
\end{proof}

Let $\widetilde{H}_i\subset X_{n+1}^{1,n}$ be the strict transform of $\pi_2^{-1}(\{y_i = 0\})$. For $p\in\mathbb{P}^1\setminus\{\pi_1(p_1),\dots,\pi_1(p_{n+1})\}$ the fiber $F_p = \widetilde{\pi}_1^{-1}(p)$ is isomorphic to $\mathbb{P}^n$ while $F_{p_i} = \widetilde{\pi}_1^{-1}(p_i)$ is the union of two irreducible components $D_{i,1}\cup\ D_{i,2}$ where $\ D_{i,1}$ is isomorphic to $\mathbb{P}^n$ blown-up in a point with exceptional divisor $\overline{E}_i\cong\mathbb{P}^{n-1}$, $\ D_{i,2} = E_i$ and $D_{i,1}\cap D_{i,2} = \overline{E}_i$. Recall that $E_i\cong\mathbb{P}^n$ and denote by $R_i\subset E_i$ the union of the $T$-invariant divisors of the lifting of (\ref{action}) restricted to $E_i$. Then 
$$X_0 = X_{n+1}^{1,n}\setminus\left(\bigcup_{i=0}^n\widetilde{H}_i \cup \bigcup_{j=1}^{n+1}R_j\right)\subset X_{n+1}^{1,n}$$
is the $T$-invariant open subset of points of $X_{n+1}^{1,n}$ having isotropy group of dimension zero. Hence, the $\widetilde{H}_i$ are the $T$-invariant prime divisors supported in $X_{n+1}^{1,n}\setminus X_0$. Moreover, note that the isotropy group of a general point of both $D_{i,1},D_{i,2}$, with respect to the lifting of (\ref{action}), is trivial. In particular, it is finite of order one. 

Now, associate to the $\widetilde{H}_i$ variables $S_i$, to the $D_{i,1}$ variables $T_{i,1}$ and to the $D_{i,2}$ variables $T_{i,2}$. For $1\leq i\leq n-1$ set $k = j+1 = i+2$ and
$$
g_i = (\beta_k\alpha_j-\beta_j\alpha_k)T_{i,1}T_{i,2} + (\beta_i\alpha_k-\beta_k\alpha_i)T_{j,1}T_{j,2} + (\beta_j\alpha_i-\beta_i\alpha_j)T_{k,1}T_{k,2}.
$$
Then we have the following result:

\begin{thm}\label{Cox_n+1}
For the Cox ring of $X_{n+1}^{1,n}$ we have the following explicit presentation
$$
\Cox(X_{n+1}^{1,n}) \cong \frac{\mathbb{C}[S_0,\dots,S_n,T_{1,1},T_{1,2},\dots,T_{n+1,1},T_{n+1,2}]}{\left\langle g_1,\dots,g_{n-1}\right\rangle}
$$
where $S_i$ is the section associated to $H_2-E_1-\dots-E_{i-1}-E_{i+1}-\dots -E_{n+1}$, $T_{i,1}$ is the section associated to $H_1-E_i$ and $T_{i,2}$ is the section associated to $E_i$. 
\end{thm}
\begin{proof}
In the previous set up it is enough to apply Theorem \ref{HS}.
\end{proof}

In the rest of this section we will derive some of consequences of Theorem \ref{Cox_n+1}.

\begin{Proposition}\label{Mov_n+1}
Consider the following divisor classes on $X_{n+1}^{1,n}$:
$$
\begin{array}{lll}
D_1 & = & H_{1};\\ 
D_h & = & H_{2}-E_{i_{1}}-\dots-E_{i_{h}};\\
D_{i_1,\dots,i_n} & = & H_{1}+H_{2}-E_{i_{1}}-\dots-E_{i_{n}};\\
D_{1,\dots,n+1} & = & H_{1}+H_{2}-E_{1}-\dots-E_{n+1};\\
D_k & = & kH_{2}-kE_{i_{1}}-\dots-kE_{i_{(n-k)}}-(k-1)E_{i_{(n-k+1)}}-\dots-(k-1)E_{i_{n+1}};\\
D_{n} & = & nH_{2}-(n-1)E_{1}-\dots-(n-1)E_{(n+1)}.
\end{array} 
$$
The movable cone of $X^{1,n}_{n+1}$ is given by
$$\Mov(X_{n+1}^{1,n})=\langle D_1,D_h,D_{i_1,\dots,i_n},D_{1,\dots,n+1},D_k,D_n\rangle$$
for $2\leq k\leq n-1$, $0\leq h\leq n-1$ and $\{{i_1,\dots,i_n}\}\subset\{1,\dots,n+1\}$.

Furthermore, let $\Mov_{1}(X_{n+1}^{1,n})$ be the cone of curves covering a divisor in $X_{n+1}^{1,n}$. Then 
$$
\Mov_{1}(X_{n+1}^{1,n}) = \left\langle h_1,h_2-e_i,h_1+(n-1)h_2-e_{i_1}-\dots -e_{i_n},e_i\right\rangle
$$
for $1\leq i\leq n+1$, $\{i_1,\dots,i_n\}\subset\{1,\dots,n+1\}$.
\end{Proposition}
\begin{proof}
Set $\mathcal{C} = \langle D_1,D_h,D_{i_1,\dots,i_n},D_{1,\dots,n+1},D_k,D_n\rangle$. Since any divisor appearing in $\mathcal{C}$ is movable we have that $\mathcal{C}\subset \Mov(X_{n+1}^{1,n})$. Then taking the dual cones we get $\Mov(X_{n+1}^{1,n})^{*}\subset \mathcal{C}^{*}$. 

Now, $\Mov_{1}(X_{n+1}^{1,n})\subset \Mov(X_{n+1}^{1,n})^{*} \subset \mathcal{C}^{*}$. If we prove that any extremal ray of $\mathcal{C}^{*}$ corresponds to a curve that covers a divisor in $X_{n+1}^{1,n}$ we would have that $\Mov_{1}(X_{n+1}^{1,n})=\Mov(X_{n+1}^{1,n})^{*}=\mathcal{C}^{*}$ and hence $\Mov(X_{n+1}^{1,n}) = \mathcal{C}$ concluding the proof.

A straightforward computation shows that
$$
\mathcal{C} = \left\langle h_1,h_2-e_i,h_1+(n-1)h_2-e_{i_1}-\dots -e_{i_n},e_i\right\rangle
$$
for $1\leq i\leq n+1$, $\{i_1,\dots,i_n\}\subset\{1,\dots,n+1\}$. Clearly, the curves with classes $h_1,h_2-e_i,e_i$ cover divisors in $X_{n+1}^{1,n}$. Now, consider the projections in $\mathbb{P}^n$ of the blown-up points $p_{i_1},\dots,p_{i_n}$ and let $H\subset\mathbb{P}^n$ be the hyperplane spanned by them. Let $X_{n}^{1,n-1}\subset X_{n+1}^{1,n}$ be the strict transform of $\mathbb{P}^1\times H\subset\mathbb{P}^1\times \mathbb{P}^n$. To conclude it is enough to note that the curves of class $h_1+(n-1)h_2-e_{i_1}-\dots -e_{i_n}$ cover $X_{n}^{1,n-1}$.
\end{proof}

\begin{Remark}
Let $X$ be a $\mathbb{Q}$-factorial projective variety. In general $\Mov_{1}(X) \subsetneqq \Mov(X)^{*}$. However, by Proposition \ref{Mori_n+1} we have that $$\Mov_{1}(X_{n+1}^{1,n}) = \Mov(X_{n+1}^{1,n})^{*}.$$ 
For further information on the behavior of these cones for general projective varieties we refer to \cite{Pay06}, \cite{BDPS23}.
\end{Remark}

\begin{Remark}
We developed Magma \cite{Magma97} scripts computing the nef cones of $X^{1,n}_{n+1},X^{1,n}_{n+2}$ and of $X^{1,n}_{n+3}$ for $n\leq 4$, as well as the Mori chamber decomposition of $X^{1,n}_{n+1}$. A Magma library containing the scripts can be downloaded at the following link:
\begin{center}
\url{https://github.com/msslxa/Cox-rings-of-blow-ups-of-multiprojective-spaces}
\end{center}
We managed to compute the Mori chamber decomposition of $X^{1,n}_{n+1}$ for $n = 2,3,4$, and we got $92,550$ and $6307$ chambers respectively. Finally, we would like to mention that the Mori chamber decomposition of $X^{1,2}_6$ has been fully computed by T. Grange in \cite[Chapter 3]{Gra22}.
\end{Remark}

\bibliographystyle{amsalpha}
\bibliography{Biblio}
\end{document}